\documentclass[conference]{IEEEtran}
\IEEEoverridecommandlockouts
\def\BibTeX{{\rm B\kern-.05em{\sc i\kern-.025em b}\kern-.08em
    T\kern-.1667em\lower.7ex\hbox{E}\kern-.125emX}}

\usepackage{graphicx}
\usepackage{amssymb}
\usepackage{amsmath}
\usepackage{siunitx}
\usepackage{url}
\usepackage{bm}
\usepackage{multirow}
\usepackage{tikz,pgfplots}
\usepgfplotslibrary{external}
\usetikzlibrary{spy}
\pgfplotsset{compat=newest}
\usepackage{xcolor}
\definecolor{r}{rgb}{0.6350, 0.0780, 0.1840}
\definecolor{g}{rgb}{0.4660, 0.6740, 0.1880}
\definecolor{b}{rgb}{0.3010, 0.7450, 0.9330}

\usepackage{float}

\usepackage{amsthm}

\newtheorem{mydefn}{Definition}
\newtheorem{lem}{Lemma}

\newtheorem{prob}{Problem}
\newtheorem{assm}{Assumption}

\usepackage{algorithm}
\usepackage{algpseudocode}

\algdef{SE}[DOWHILE]{Do}{doWhile}{\algorithmicdo}[1]{\algorithmicwhile\ #1}%

\newcommand{\Prob}{\mathbb{P}}
\newcommand{\R}{\mathbb{R}}
\newcommand{\Ex}{\mathbb{E}}

\title{Approximate Quantiles for Stochastic Optimal Control of LTI Systems with Arbitrary Disturbances}
\author{Shawn Priore, Christopher Petersen, and Meeko Oishi 
    \thanks{This material is based upon work supported by the National Science Foundation under NSF Grant Number CMMI-2105631.  Any opinions, findings, and conclusions or recommendations expressed in this material are those of the authors and do not necessarily reflect the views of the NSF.  
        \newline \indent
        Shawn Priore and Meeko Oishi are with Electrical and
        Computer Engineering, University of New Mexico, Albuquerque, NM; e-mail: \texttt{shawnpriore@unm.edu, oishi@unm.edu} (corresponding author).
        \newline\indent
        Christopher Petersen is with the Air Force Research Lab, Space Vehicles Directorate, Albuquerque, NM.
    }
}

\begin{document}
\maketitle

\begin{abstract}
    We propose a method for open-loop stochastic optimal control of LTI systems based on Taylor approximations of quantile functions.  This approach enables efficient computation of quantile functions that arise in chance constrained reformulations.  We are motivated by multi-vehicle planning problems for LTI systems with norm-based collision avoidance constraints, and polytopic feasibility constraints. Respectively, these constraints can be posed as reverse-convex and convex chance constraints that are affine in the control and disturbance. We show for constraints of this form, piecewise affine approximations of the quantile function can be embedded in a difference-of-convex program that enables use of conic solvers. We demonstrate our method for multi-satellite coordination with Gaussian and Cauchy disturbances, and provide a comparison with particle control.
\end{abstract}

\section{Introduction}\label{sec:intro}

As missions with multiple space vehicles become commonplace, new technologies are required for effective autonomous operation that enable coordination amongst multiple vehicles in a challenging environment, despite limited resources (such as fuel). 
Autonomy for spacecraft must accommodate the need to plan and optimize under uncertainty, which may arise due to modeling inaccuracies, nonlinearities in sensing and estimation processes, and actuation mechanisms.  
Many of these uncertainties may be stochastic, but may not necessarily follow Gaussian distributions.
%, and could even be heavy-tailed.  
Constructing optimal controllers to ensure collision avoidance and performance constraints, despite such stochasticity, requires accurate assessment of risk. In this paper, we seek to construct algorithmically efficient solutions to stochastic optimization problems for cooperative, multi-vehicle problems in potentially non-Gaussian environments.   

Algorithms for stochastic optimization often face significant computational hurdles that can create undesirable trade-offs with accuracy \cite{mesbah2016stochastic}, particularly in systems with limited computation. Particle approaches \cite{calafiore2006scenario,blackmore2011chance} have employed sample reduction techniques for convex \cite{campi2011sampling,care2014fast} and non-convex \cite{Campi2018TAC} problems, but still are subject to tradeoffs between accuracy and computational burden.  Approaches that rely upon moments \cite{nemirovski2006convex,calafiore2006distributionally,paulson2017stochastic} may create excessive conservativism, and typically require an iterative approach to controller synthesis and risk allocation, to circumvent non-convexity that arises in the process of separating joint chance constraints into individual chance constraints via Boole's inequality \cite{oldewurtel2014stochastic,ono2008iterative,vitus_feedback_2011}.  Recent work has employed Fourier transforms in combination with piecewise affine approximations \cite{Sivaramakrishnan2021TAC, vinod2019piecewise}, to evaluate chance constraints without quadrature for linear time-invariant (LTI) systems with disturbance processes that have log-concave probability density functions (pdf). 

Our approach to constrained stochastic optimization of LTI systems with potentially non-Gaussian disturbances is based on approximations of a quantile function.  We consider multi-vehicle planning problems with two types of constraints: a) norm-based collision avoidance constraints, and b) polytopic feasibility constraints.  These forms readily arise when vehicles must avoid each other, as well as static obstacles in the environment, while remaining in some desirable polytopic set and reaching a desired convex target set.   We show that these constraints can be reduced to chance constraints that are affine in the control input and disturbance.  The norm-based constraints yield reverse convex constraints and the feasibility constraints yield convex constraints.  In both cases, assessment of a quantile function, the inverse of the cumulative distribution function (cdf), is necessary to evaluate these constraints.  However, quantile functions are notoriously difficult to compute.

Our approach is to construct a Taylor series approximation of the quantile function that is amenable to arbitrary distributions.  We generate an affine approximation of the quantile by evaluating the Taylor series approximation at regular intervals to yield a piecewise affine constraint that can be embedded within a standard difference-of-convex programming framework \cite{boyd_dc_2016}.  We employ an iterative approach as in \cite{ono2008iterative}, \cite{PrioreACC21}, to allocate risk and synthesize an optimal control. Although iterative, our approach can be considerably faster than a particle approach because it exploits convexity.
{\em The main contribution of this paper is the construction of a first-order quantile approximation that enables efficient evaluation of chance constraints.}  Our approach relies upon affine structure in the collision avoidance and feasibility chance constraints.

The paper is organized as follows. Section \ref{sec:prelim} provides mathematical preliminaries and formulates the optimization problem. Section \ref{sec:methods} reformulates the chance constraints by approximating the quantile function.  Section \ref{sec:results} demonstrates our approach on two multi-satellite rendezvous problems, and Section \ref{sec:conclusion} provides concluding remarks.

\section{Preliminaries and problem formulation} \label{sec:prelim}

We denote the interval that enumerates all natural numbers from $a$ to $b$, inclusively, as $\mathbb{N}_{[a,b]}$.  Random vectors are indicated with a bold case $\boldsymbol v$ and non-random vectors with an overline $\overline{v}$. We presume $I_{n}$ represents an identity matrix of size $n$, and $0_{n\times m}$ represents a $n\times m$ matrix of zeros. We denote the 2-norm of a matrix or vector by $\|\cdot\|$. For a random variable, we denote its pdf as $\phi$, its cdf as $\Phi$, and its quantile function as $\Phi^{-1}$. 

\subsection{Problem Formulation}
We are motivated by problems in multi-vehicle stochastic optimal control. We consider a discrete-time LTI system given by 
\begin{equation}
    \boldsymbol x(k+1) = A \boldsymbol x(k) + B \overline{u}(k) + \boldsymbol{w}(k) \label{eq:system}
\end{equation}
with state $\boldsymbol{x}(k) \in \R^n$, input $\overline{u}(k) \in \mathcal{U} \subset \R^m$, disturbance $\boldsymbol{w}(k) \in \R^n$ that is a 
%independently but not necessarily identically distributed
random variable, and initial condition $\overline{x}(0)$.  We presume 
% , and time index $k \in \mathbb{N}_{[0,N]}$ . 
% Initial condition $\overline{x}(0)$ is known, and the set 
$\mathcal{U}$ is a convex polytope and that the system evolves over a finite time horizon of $N \in \mathbb{N}$ steps. 

% With a finite time horizon $N \in \mathbb{N}$, we can exploit the linearity of the system as an affine summation of a transformed initial state, a concatenated input vector, and a concatenated disturbance vector:
We rewrite the dynamics at time $k$ as
\begin{equation} \label{eq:lin_dynamics}
    \boldsymbol x(k) = A^k \overline{x}(0) + \mathcal{C}_u(k) \overline{U} + \mathcal{C}_w(k) \boldsymbol{W}
\end{equation}
with 
\begin{subequations}
\begin{alignat}{2}
    \overline{U} =& \left[ u(0)^\top \:\:\ldots \:\:u(N-1)^\top\right]^\top &&\in \mathcal{U}^{N} \\
    \boldsymbol{W} =& \left[ \boldsymbol{w}(0)^\top \:\:\ldots \:\:\boldsymbol{w}(N-1)^\top\right]^\top &&\in \mathbb{R}^{Nn} \\
    \mathcal{C}_u(k) = & \left[ A^{k-1}B \:\: \ldots \:\: AB \:\: B \:\: 0_{n \times (N-k)m} \right] && \in \R^{n \times Nm} \\
    \mathcal{C}_w(k) = & \left[ A^{k-1} \:\:\ldots \:\:A \:\:I_n \:\:0_{n \times (N-k)n} \right] && \in \R^{n \times Nn}
\end{alignat}
\end{subequations}    
% for all $k \in \mathbb{N}_{[1,N]}$.

We consider a planning context, in which (\ref{eq:system}) captures the evolution of $v$ vehicles in a bounded region, with state $\boldsymbol x_i$ and concatenated input $\overline{U}_i$ for vehicle $i$. 
% Each vehicle will be denoted with an integer subscript on it's respective state and input, $\boldsymbol x_1$ and $\overline{U}_1$, for example. Here, we control each vehicle and each vehicle is controlled independently.  
We presume 
% that there may be 
desired target sets that vehicles must reach, known and static obstacles that vehicles must avoid, as well as the need for collision avoidance between vehicles, all with desired likelihoods.  
% That is, for some vehicl  
% We capture these constraints probabilistically as 
\begin{subequations}\label{eq:constraints}
    \begin{align}
    \Prob \left\{\boldsymbol x_i(k)  \in  \mathcal T_{i,k} \right\} & \geq  1\!-\!\alpha_{\mathcal T} \label{eq:constraint_t}\\
    \Prob \left\{\boldsymbol \| S \boldsymbol x_i(k) \!-\! S \boldsymbol o \| \geq r \right\} &\geq 1\!-\!\alpha_o, \: \forall i \in \mathbb N_{[1, v]} \label{eq:constraint_o}\\
    \Prob \left\{\boldsymbol \| S \boldsymbol x_i(k) \!- \!S \boldsymbol x_j (k) \| \geq r \right\} &\geq 1\!-\!\alpha_r, \: \forall i \neq j \in \mathbb N_{[1, v]} \label{eq:constraint_r}
    \end{align}
\end{subequations}
We presume convex, compact, and polytopic sets $ \mathcal T_k \subseteq \mathbb R^n$, known matrix $S \in \mathbb R^{q \times m}$, positive scalar $r \in \mathbb R_+$, static object locations $\boldsymbol o \in \R^n$, 
and probabilistic violation thresholds $\alpha_{\mathcal T}, \alpha_{o}, \alpha_r, \in (0,1)$.
% The probabilistic violation thresholds for each of the three types of constraints are described by $\alpha_{\mathcal T}, \alpha_{o}, \alpha_r, \in (0,1)$. 
% Note that through the properties of the intersection operator we can combine constraints \eqref{eq:constraints} within a single probability function and under a single probabilistic violation threshold.

We seek to minimize a convex performance objective $J: \mathcal{X}^{N \times v} \times \mathcal{U}^{N \times v} \rightarrow \R$. 
%We seek to minimize the performance objective with respect to (\ref{eq:constraints})
%probabilistic constraints of the following form,
\begin{subequations}\label{prob:big_prob_eq}
    \begin{align}
        \underset{\overline{U}_1, \dots, \overline{U}_v}{\mathrm{minimize}} \quad & J\left(
        \boldsymbol{X}_1, \ldots, \boldsymbol X_v,  \overline{U}_1, \dots, \overline{U}_v\right)  \\
        \mathrm{subject\ to} \quad  & \overline{U}_1, \dots, \overline{U}_v \in  \mathcal U^N,  \\
        & \text{Dynamics } \eqref{eq:lin_dynamics} \text{ with }
        \overline{x}_1(0), \dots, \overline{x}_v(0)\\
        %& \qquad \overline{x}_1(0), \dots, \overline{x}_v(0) \nonumber \\
        & \text{Probabilistic constraints  \eqref{eq:constraints}}
    \end{align}
\end{subequations}
where $\boldsymbol X_i = \begin{bmatrix} \boldsymbol x_i^{\top}(1) & \ldots & \boldsymbol x_i^{\top}(k) \end{bmatrix}^{\top}$ is the concatenated state vector for vehicle $i$.

We first note that each constraint in \eqref{eq:constraints} can be rewritten in one of the two following forms, which are convex in the control input and affine in a random variable (as shown in Appendix \ref{sec:reformulation}).

\noindent\begin{subequations}
\begin{align}
    \Prob \left\{ \bigcap_{j=1}^v \bigcap_{i=1}^{n_j} f_i(\overline{x}_j(0), \overline{U}_j) + g_i \eta_i \leq c_i \right\} & \geq 1-\alpha \label{eq:prob_constraints_convex} \\
    \Prob \left\{ \bigcap_{j=1}^v \bigcap_{i=1}^{n_j} f_i(\overline{x}_j(0), \overline{U}_j) - g_i \eta_i \geq c_i \right\} & \geq 1-\alpha \label{eq:prob_constraints_reverse_convex}
\end{align}
\end{subequations}
The function $f_i(\cdot): \mathcal{X} \times \mathcal{U}^N \rightarrow \R $ is convex, $g_i \in \R_+$ is a scalar, and $\eta_i$ is a real and continuous random variable that is a function of the disturbance.  
% where the state constraint is captured by $f_i(\cdot): \mathcal{X} \times \mathcal{U}^N \rightarrow \R $, and the stochastic element is captured by $g_i \eta_i$, where $g_i \in \R \setminus \{0\}$ is a scalar, and $\eta_i$ is a real and continuous random variable.  
We presume $n_j$ is the number of scalar constraints imposed on vehicle $j$, $c_i$ is a constant, and $\alpha$ is a probabilistic violation threshold.

\begin{assm}\label{assm:1}
The pdf of the random variable $\eta_i$ can be differentiated at least $n_d$ times, and the quantile of $\eta_i$ must be convex and non-negative in the region $\mathcal A = [1-\alpha, 1]$.
\end{assm}

This assumption is not overly restrictive, as it can 
% $\eta_i$ can 
be met by most distributions. 
Differentiability is needed for the Taylor series approximation of the quantile; fewer derivatives means a coarser approximation.  
% depending on their parameterization. 
% We note that in practice, the pdf of $\eta_i$ need only be sufficiently differentiable on $\mathcal A = [\Phi^{-1}_{\eta_i}(1-\alpha), \infty)$, because [[??]].
% $[a, \infty)$ where $a = \Phi^{-1}_{\eta_i}(1-\alpha)$. 
% A good rule of thumb for the distribution meeting the necessary requirements is: 
Convexity over this range is met when 
1) the first derivative of the pdf is strictly negative on $\mathcal{B} \equiv \left\{x \mid \Phi_{\eta_i}(x) \in \mathcal A \right\}$, and 2) the pdf converges to zero as $x$ increases on $\mathcal{B}$.
% \begin{enumerate}
%     \item the pdf monotonically decreases on $\mathcal A$, and
%     \item the pdf converges to zero  as $x$ increases on $\mathcal A$
% \end{enumerate} 
The intuition behind these criteria is that when both conditions are met, the cdf will be strictly concave on $\mathcal{B}$, and hence, the quantile will be strictly convex.  
However, we note that not all distributions 
%that meet these two properties 
will have a convex quantile in $\mathcal{A}$: 
Consider the Beta distribution with shape parameters $(\alpha, \beta)$ both less than one, which results in a bi-modal distribution with modes at both ends of the support. 
% A good example of a distribution that doesn't meet this criteria is a Beta distribution with shape parameters $(\alpha, \beta)$ both less than one. This results in a bi-modal distribution with modes at both ends of the support.

The reformulation \eqref{eq:prob_constraints_convex}-\eqref{eq:prob_constraints_reverse_convex} requires the quantile evaluations be non-negative to tighten the probabilistic constraints. For constraints like \eqref{eq:constraint_t}, many models assume a symmetric distribution and $\Ex[\eta_i]=0$. This implies the quantile is non-negative in the convex region. Similarly, the use of distance metrics in \eqref{eq:constraint_o}-\eqref{eq:constraint_r} imply $\eta_i$ will be strictly positive. In the event that this is not the case, a slight modification to \eqref{eq:prob_constraints_convex}-\eqref{eq:prob_constraints_reverse_convex} and Assumption \ref{assm:1} can be made to maintain tightening of the constraint.

\begin{assm}\label{assm:2}
The random variable $\eta_i$ has a known quantile, $\Phi^{-1}_{\eta_i}(p)$, for some $p \in (0,1)$.
\end{assm}

This assumption is easily met by symmetric distributions;
% In the case of symmetric distributions, this assumption is easily met. 
many have either a location parameter that represents the median or an easy way to find the median. For distributions on a semi-infinite support or that are skewed, satisfying this assumption may be more difficult.  Approximations via brute force or other methods may be distribution dependent.
% available and depending on the characteristics of the distribution. 
In practice, it may be sufficient to choose $\Phi^{-1}_{\eta_i}(1-\varepsilon)$ to be some value approaching the upper end of the support, setting $\varepsilon$ to be arbitrarily small. 

\begin{prob} \label{prob:initial}
    Solve the optimization problem
    \begin{subequations}\label{prob:initial_eq}
        \begin{align}
        \underset{\overline{U}_1, \dots, \overline{U}_v}{\mathrm{minimize}} \quad & J\left(
        \boldsymbol{X}_1, \ldots, \boldsymbol X_v,  \overline{U}_1, \dots, \overline{U}_v\right)  \\
            \mathrm{subject\ to} \quad  & \overline{U}_1, \dots, \overline{U}_v \in  \mathcal U^N,  \\
            & \text{\normalfont Dynamics \eqref{eq:lin_dynamics} with } \overline{x}_1(0), \dots, \overline{x}_v(0)  \\
            & \text{\normalfont Probabilistic constraints \eqref{eq:prob_constraints_convex}-\eqref{eq:prob_constraints_reverse_convex}} \label{prob:initial_eq_prob_constraints}
        \end{align}
    \end{subequations}
    with open loop control $\overline{U}_1, \dots, \overline{U}_v \in  \mathcal U^N$, for probabilistic violation thresholds small enough to maintain convexity, under Assumptions 1 and 2. 
\end{prob} 
% The approach presented here can be generalized to accommodate variations on Problem \ref{prob:initial}: 1) Inclusion of uncontrollable objects with the same dynamics about a common origin or 2) in some cases, stochastic initial conditions. 

% where $n_j$ is the number of constraints imposed on vehicle $j$, $f_i(\cdot): \mathcal{X} \times \mathcal{U}^N \rightarrow \R $ is a non-stochastic convex function of the state, $g_i \in \R$ is a scalar, $\eta_i$ is a real and continuous random variable with a known distribution, a sufficiently differentiable pdf, and the quantile must be convex in the region $[\alpha, 1]$, $c_i$ is a constant, and $\alpha$ is a safety threshold that is sufficiently small enough to maintain convexity of the optimization program. 

% In many applications, the use of linear dynamics in planning problems facilitate the separation of the stochastic element, $g_i \eta_i$, from the state constraint, $f_i(\overline{x}_j(0), \overline{U}_j)$. Specifically, for \eqref{eq:prob_constraints_convex}, $f_i(\overline{x}_j(0), \overline{U}_j)$ can represent state or terminal constraints, or actuation constraints when the actuation is perturbed. Similarly, for \eqref{eq:prob_constraints_reverse_convex}, $f_i(\overline{x}_j(0), \overline{U}_j)$ can represent inter-vehicle distance or obstacle avoidance constraints. As a demonstration, we derive the obstacle avoidance constraint, \eqref{eq:prob_constraints_reverse_convex}, for an an vehicle and some uncontrollable object. 

The main challenge in solving Problem \ref{prob:initial} is assuring \eqref{prob:initial_eq_prob_constraints}.

\section{Methods}\label{sec:methods}

To solve Problem \ref{prob:initial}, we employ a standard risk allocation framework in conjunction with a quantile reformulation.  We then 
% Solving Problem \ref{prob:initial} will require multiple steps to satisfy \eqref{prob:initial_eq_prob_constraints}. First, we use a quantile reformulation with risk allocation via Boole's inequality to reformulate the joint chance constraint as a series of individual constraints. Next, we 
% compute a numerical approximation 
approximate the quantile function over its convex region, 
% and generate a affine approximation such that we have continuous constraints. 
via piecewise affine constraints.
Lastly, we employ difference-of-convex programming to iteratively solve reverse convex constraints to a local optimum. These reformulations enable solution via a series of quadratic programs. 

\subsection{Quantile Reformulation}
% The reformulations for \eqref{eq:prob_constraints_convex} and \eqref{eq:prob_constraints_reverse_convex} are 
% similar and follow the same line of reasoning. 
First, consider the reformulation of \eqref{eq:prob_constraints_convex}.
We take the complement of \eqref{eq:prob_constraints_convex} such that the probability function consists of a union of events,
\begin{equation}\label{eq:quantile_reform_comp}
    \Prob \left\{ \bigcup_{j=1}^{v} \bigcup_{i=1}^{n_j} \left[f_i(\overline{x}_j(0), \overline{U}_j) + g_i \eta_i \leq c_i \right]^c \right\} \leq \alpha 
\end{equation}
To take the complement, we reverse the sign of the inequality. Next, we implement Boole's inequality to create an upper bound for the original probability,
\begin{align}\label{eq:quantile_reform_boole}
     &\Prob \left\{\bigcup_{j=1}^{v} \bigcup_{i=1}^{n_j} f_i(\overline{x}_j(0), \overline{U}_j) + g_i \eta_i \geq c_i \right\} \\
     & \qquad \leq \sum_{j=1}^{v} \sum_{i=1}^{n_j} \Prob \left\{ f_i(\overline{x}_j(0), \overline{U}_j) + g_i \eta_i \geq c_i  \right\} \label{eq:boole_right}
\end{align}
Using the approach in \cite{ono2008iterative}, we introduce variables 
% By introducing a new series of variables 
$\underline{\omega}_{ij}$ 
% and taking the complement, we can 
to allocate risk to each of the individual probabilities
% \cite{ono2008iterative}, 
\begin{subequations}\label{eq:quantile_reform_new_var}
\begin{align}
     \Prob \left\{ f_i(\overline{x}_j(0), \overline{U}_j) + g_i \eta_i \leq c_i  \right\} &\geq 1-\underline{\omega}_{ij} \label{eq:quantile_orig} \\
     \sum_{j=1}^{v} \sum_{i=1}^{n_j} \underline{\omega}_{ij} &\leq \alpha \label{eq:quantile_reform_new_var_2}\\
     \underline{\omega}_{ij} & \geq 0 \label{eq:quantile_reform_new_var_3}
\end{align}
\end{subequations}
% Next we rearrange the probability function such that we can invert it,
% \begin{align}\label{eq:quantile_reform_get_single_var}
    % &\Prob\left\{ f_i(\overline{x}_j(0), \overline{U}_j) + g_i \eta_i \leq c_i  \right\} \\
    % & \qquad = \Prob \left\{ \eta_i \leq \frac{1}{g_i}\left(c_i - f_i(\overline{x}_j(0), \overline{U}_j) \right) \right\} \nonumber
% \end{align}
By inverting the argument of (\ref{eq:quantile_orig}), we obtain
% We see that the probability is now a cdf of the random variable $\eta_i$ and is invertable,
\begin{equation}
\begin{array}{rrl}
    &\Prob \left\{ \eta_i \leq \frac{1}{g_i}\left(c_i - f_i(\overline{x}_j(0), \overline{U}_j) \right) \right\} & \geq 1-\underline{\omega}_{ij} \\
    \Leftrightarrow &\frac{1}{g_i}\left(c_i - f_i(\overline{x}_j(0), \overline{U}_j) \right) & \geq \Phi^{-1}_{\eta_i}\left(1-\underline{\omega}_{ij} \right)
\end{array}
\label{eq:quantile_opt}
\end{equation}
% Through algebraic manipulation, we end with a series of constraints in the following form subject to \eqref{eq:quantile_reform_new_var_2}-\eqref{eq:quantile_reform_new_var_3}, 
Rearranging (\ref{eq:quantile_opt}), we obtain
\begin{equation}
     f_i(\overline{x}_j(0), \overline{U}_j)  \leq c_i - g_i\left(\Phi^{-1}_{\eta_i}\left(1-\underline{\omega}_{ij} \right)\right)
\end{equation}

The reformulation of \eqref{eq:prob_constraints_reverse_convex} proceeds similarly, and results in reverse convex constraints.  

\begin{mydefn}[Reverse convex constraint]
A reverse convex constraint is the complement of a convex constraint, that is, $f(x) \geq c$ for a convex function $f: \R \rightarrow \R$ and a scalar $c \in \mathbb R$.
\end{mydefn}
% Note that in the reformulation step presented in \eqref{eq:quantile_reform_get_single_var} the negative sign preceding the function $g_i$ changes the inequality. 
By combining the reformulations \eqref{eq:prob_constraints_convex}-\eqref{eq:prob_constraints_reverse_convex}, we obtain
\begin{subequations}\label{eq:quantile_final}  
\begin{align} 
     f_i(\overline{x}_j(0), \overline{U}_j)  &\leq c_i - g_i \left(\Phi^{-1}_{\eta_i}\left(1-\underline{\omega}_{ij} \right)\right) \label{eq:quantile_reform_final_1}\\
     \textstyle\sum_{j=1}^{v} \sum_{i=1}^{n_j} \underline{\omega}_{ij} &\leq \alpha  \label{eq:quantile_reform_final_2}\\
     \underline{\omega}_{ij} & \geq 0 \label{eq:quantile_reform_final_3}\\ 
     f_i(\overline{x}_j(0), \overline{U}_j)  &\geq c_i +g_i \left(\Phi^{-1}_{\eta_i}\left(1+\underline{\beta}_{ij} \right)\right) \label{eq:quantile_reform_final_4}\\
     \textstyle\sum_{j=1}^{v} \sum_{i=1}^{n_j} \underline{\beta}_{ij} &\leq \alpha \label{eq:quantile_reform_final_5}\\
     \underline{\beta}_{ij} & \geq 0\label{eq:quantile_reform_final_6}
\end{align}
\end{subequations}

\begin{lem} \label{lem1}
For the controller $\overline{U}_1, \dots, \overline{U}_v$, if there exists risk allocation variables $\underline{\omega}_{ij}$ satisfying \eqref{eq:quantile_reform_final_2}-\eqref{eq:quantile_reform_final_3} for constraints in the form of \eqref{eq:quantile_reform_final_1} and risk allocation variables $\underline{\beta}_{ij}$ satisfying \eqref{eq:quantile_reform_final_5}-\eqref{eq:quantile_reform_final_6} for constraints in the form of \eqref{eq:quantile_reform_final_4}, then $\overline{U}_1, \dots, \overline{U}_v$ satisfy \eqref{prob:initial_eq_prob_constraints}.
\end{lem}

\begin{proof}
Satisfaction of \eqref{eq:quantile_reform_final_2}-\eqref{eq:quantile_reform_final_3} and \eqref{eq:quantile_reform_final_5}-\eqref{eq:quantile_reform_final_6} implies \eqref{eq:boole_right} meets the probabilistic violation threshold of $1-\alpha$. Boole's inequality and De Morgan's laws guarantee \eqref{prob:initial_eq_prob_constraints} is satisfied.
\end{proof}

The constraint \eqref{eq:quantile_reform_final_1} is convex in $\overline U$, however 
\eqref{eq:quantile_reform_final_4} is reverse convex.  
% with respect to the controller $\overline{U}_j$. 
Additionally, while Assumption \ref{assm:1} guarantees the convexity of \eqref{eq:quantile_reform_final_1}, the expressions $\Phi^{-1}_{\eta_i}(1-\underline{\omega}_{ij})$ 
and $\Phi^{-1}_{\eta_i}(1-\underline{\beta}_{ij})$
are non-conic and cannot be readily handled by off-the-shelf solvers. 

% In Section \ref{sub:quantiles}, we discuss how we employ a piecewise affine approximation of $\Phi^{-1}_{\eta_i}(1-\underline{\omega}_{ij})$ on the range $\underline{\omega}_{ij} \in [0,\alpha]$. However, constraints \eqref{eq:quantile_reform_final_4}-\eqref{eq:quantile_reform_final_6} are reverse convex with respect to the controller $\overline{U}_j$. While we employ the same procedure to handle the non-conic nature of $\Phi^{-1}_{\eta_i}(1-\underline{\beta}_{ij})$, we need to add an additional step to handle these constraints. In Section \ref{sub:dcp}, we discuss outline how we use a difference of convex program to solve Problem \ref{prob:initial}. When using a difference of convex program, Lemma \ref{lem1} guarantees a feasible but locally optimal solution.

\subsection{Quantile Approximation} \label{sub:quantiles}
The quantile for many
%most 
continuous random variables
% , $\Phi^{-1}_{\eta_i}(\cdot)$, 
does not have a closed form, and
% Many of these distributions have unique 
brute force numerical approximations may be costly to compute. 
Approximation methods are typically tailored to specific distributions \cite{abramowitz1964handbook}, \cite{kafadar1988bidec}, \cite{wichura1988algorithm},
although some recent approaches have focused on
% . In recent years there has been a push toward more 
generic methods to approximate quantile functions of arbitrary distributions.  

We use an approach that relies on a Taylor series expansion of the quantile \cite{Yu2017}. 
% We outline the method here.
For a random variable $X$, and an initial evaluation point $\Phi^{-1}_{X}(p_0)$ for $p_0 \in (0,1)$, \cite{Yu2017} proposes an iterative process that
% for approximating the quantile function involves 
evaluates a finite Taylor series expansion at points that are an interval $h \in \mathbb R$ apart.  With $n_d+1$ Taylor series terms, a quantile approximation at $p_{c+1} = p_c + h$ is described by
\begin{align}\label{eq:taylor_quantile}
    \hat{\Phi}^{-1}_{X}(p_{c+1}) = &\; \Phi^{-1}_{X}(p_c) \\
    &+\sum_{d=1}^{n_d+1} (-1)^{d}  \left. \frac{\partial^d \Phi^{-1}_{X}(p)}{(\partial \gamma)^d} \right|_{p = p_{c}} \cdot \frac{\log(p_{c+1}/p_c)^d}{d!} \nonumber
\end{align}  
where $\gamma =- \log(p)$ is a variable substitution used for numerical tractability.  Typically, $n_d=3$ or 4 derivatives are sufficient, and steps $c$ are computed until a predetermined terminating percentile.  

% Let $X$ be a random variable with pdf, $\phi_X (\cdot)$. Suppose $\phi_X (\cdot)$ is differentiable up to $n$ derivatives (typically 3 or 4 derivatives are sufficient), and that $\Phi^{-1}_X (p_0)$ is a known quantile value for some initial point $0 < p_0 < 1$. Define a sufficiently small step size $h$ such that $p_{1} = p_0 + h$, and define $\gamma =- \log(p_1)$. Then, 
% \begin{equation}\label{eq:taylor_quantile}
%     \hat{\Phi}^{-1}_{X}(p_1) = \Phi^{-1}_{X}(p_0) + \sum_{i=1}^n \frac{\partial^i \Phi^{-1}_{X}(p_0)}{(\partial \gamma)^i} \frac{\log(p_1/p_0)^i}{i!}
% \end{equation}  
% where $\hat{\Phi}^{-1}_{X}(\cdot)$ indicates the estimate of $\Phi^{-1}_{X}(\cdot)$. By iterating over \eqref{eq:taylor_quantile}, we estimate a series of points until we reach some desired end point $\hat{\Phi}^{-1}_{\eta_i}(p_k)$ for some $k \in \mathbb N$. 
% This also affords a straightforward expression for the 
Derivatives of the quantile are obtained via the inverse function theorem,
% Here, the derivatives of the quantile are computed via the inverse function theorem which, transcribed to our problem, states:
\begin{equation}
    \frac{\partial}{\partial \gamma} \Phi^{-1}_{X}(p) = -\frac{e^{-\gamma}}{\phi_X (p)}
\end{equation}
where the $i$\textsuperscript{th} derivative of the quantile will elicit the the $i-1$\textsuperscript{th} derivative of $\phi_X (\cdot)$. Analytical expressions for the first four derivatives are provided in \cite{Yu2017}.

The error in the approximation 
\begin{equation}\label{eq:taylor_error_eq}
    \epsilon = \Phi^{-1}_{X}(\cdot) - \hat{\Phi}^{-1}_{X}(\cdot) 
\end{equation}  
is characterized by the unused Taylor series terms, such that
% allows up to characterize $\epsilon$ as the remaining terms of the expansion in \eqref{eq:taylor_quantile}. Thus,
\begin{equation} \label{eq:taylor_error}
    \epsilon \in O\left([h / \min (p_{l-1}, p_0)]^{n_d}\right)
\end{equation} 
so that $\epsilon$ converges to $0$ as $h \rightarrow 0$ and $n_d \rightarrow \infty$ \cite{Yu2017}.

% To incorporate \eqref{eq:taylor_quantile} into \eqref{eq:quantile_final}, we will need to transform the discrete approximations into a continuous form. 
We presume a piecewise affine approximation to connect evaluation points.  However, to ensure a reasonable number of variables and constraints in the optimization, we selectively choose evaluation points, rather than connecting all points.  Given an error threshold, $\xi$, we seek a subset of $l^\ast$ affine terms, such that 
\begin{equation} \label{eq:quantile_over_error}
    \hat{\Phi}^{-1}_{X}(p_c) \leq \max_{q \in \mathbb{N}_{[1,l^\ast]}} (\underline{m}_{ijq} \, \underline{\omega}_{ij} + \underline{c}_{ijq}) \leq \hat{\Phi}^{-1}_{X}(p_c) + \xi 
\end{equation}
for slopes and intercepts $\underline{m}_{ijq}$, $\underline{c}_{ijq}$, respectively, for $\forall q \in \mathbb{N}_{[1,l^\ast]} $, as shown in Figure \ref{fig:approx_graph}. Here, $i$ and $j$ refer to the vehicle and constraint indices, respectively. We propose Algorithm \ref{algo:PWA} to compute the reduced set $\{\underline{m}_{ijq}, \underline{c}_{ijq} \mid \forall q \in \mathbb{N}_{[1,l^\ast]} \}$. Note that although the error threshold, $\xi$, is formulated with respect to the approximation (not the true quantile),  Assumption \ref{assm:1} guarantees that \eqref{eq:quantile_over_error} becomes an affine overapproximation of the true quantile as $\epsilon \rightarrow 0$. 
\begin{figure}
\centering
\includegraphics{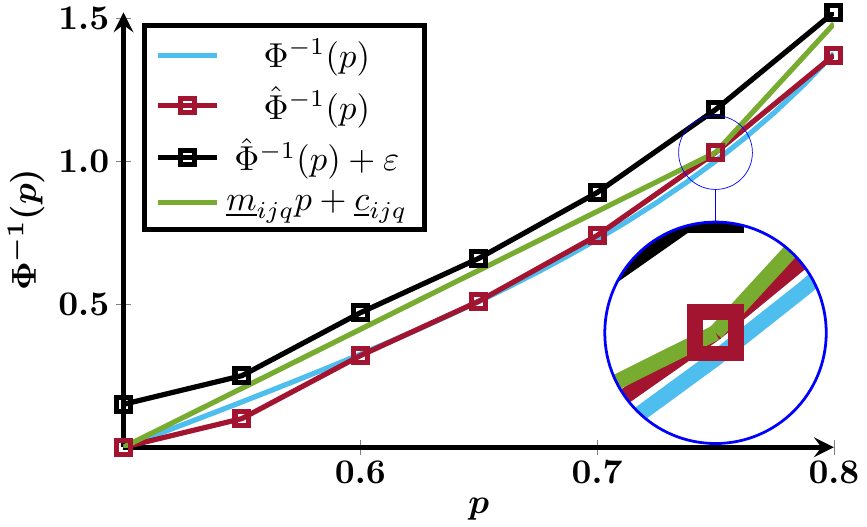}
\caption{Quantile approximation method applied to a Cauchy distribution.  The blue line represents the true quantile, the red points result from a Taylor series approximation (\ref{eq:taylor_quantile}), the black points show the error threshold $\xi$, and the green lines represent the affine approximation (\ref{eq:quantile_over_error}). 
% The represented function is the quantile of a standard Cauchy distribution.
}
\label{fig:approx_graph}
\end{figure}

\begin{algorithm}
  \caption{Computing $\{\underline{m}_{ijq}, \underline{c}_{ijq}\}$ from $\phi_{\eta_i}$}
	\label{algo:PWA}
	\textbf{Input}: The PDF of $\eta_i$, $\phi_{\eta_i}$, and its derivatives $\phi^{'}_{\eta_i}, \ldots, \phi^{(n)}_{\eta_i}$, instantiating point $p_0$, termination point $p_l$, known quantile $\Phi^{-1}_{\eta_i}(p_0)$, step size $h$, and maximum error threshold $\xi$.
	\\
	\textbf{Output}: Affine terms of $\hat\Phi^{-1}_{\eta_i}$, $\{\underline{m}_{ijq}, \underline{c}_{ijq}\}$
	\begin{algorithmic}[1]
	\For{$p_i = p_0+h$ \textbf{to} $p_l$ \textbf{by} $h$}
	\State $\mathcal P_i \gets \hat{\Phi}^{-1}(p_i)$ \Comment{Via \eqref{eq:taylor_quantile}}
	\EndFor
	\State	$i \gets 0$
	\While{$i<l$}
	\For{$j = l$ \textbf{to} $i+1$ \textbf{by} $-1$}
	\State $\underline{m} \gets \frac{\mathcal P_j-\mathcal P_i}{h(j-i)}$
	\State $\underline{c} \gets \mathcal P_i - p_i \times \underline{m}$
	\For{$y=i+1$ \textbf{to} $l-1$ \textbf{by} $1$}
	\State $\epsilon_y = \mathcal P_y - (p_y \times \underline{m} + \underline{c})$
	\State \textbf{if} $\epsilon_y > \xi$ \textbf{then next} $j$
	\EndFor
	\State $\{\underline{m}_{ijq}, \underline{c}_{ijq}\} \gets \underline{m}, \underline{c}$
	\State \textbf{Break}
	\EndFor	
	\State $i \gets j$
	\EndWhile
  \end{algorithmic}
\end{algorithm}

We reformulate \eqref{eq:quantile_reform_final_1} with the piecewise affine approximation \eqref{eq:quantile_over_error}, as
% \eqref{eq:taylor_quantile} as
\begin{subequations}\label{eq:Num_quantile_reform}
\begin{align}
     f_i(\overline{x}_j(0), \overline{U}_j)  &\leq c_i - \frac{1}{g_i}\left(\underline{s}_{ij}\right) & \label{eq:Num_quantile_reform_1}\\
     \underline{s}_{ij} & \geq \underline{m}_{ijq} \, \underline{\omega}_{ij} + \underline{c}_{ijq} & \forall q \in \mathbb{N}_{[1,l^\ast]}\\
     \textstyle\sum_{j=1}^{v}\sum_{i=1}^{n_j} \underline{\omega}_{ij} &\leq \alpha & \\
     \underline{\omega}_{ij} & \geq 0 &
\end{align}
\end{subequations}
with slack variables $\underline{s}_{ij}$.  A similar reformulation can be posed for \eqref{eq:quantile_reform_final_4}.
% and note that  \eqref{eq:Num_quantile_reform_1} has an equivalent representation for the equations in the form of \eqref{eq:quantile_reform_final_4}. 
In the limit, as \eqref{eq:quantile_over_error} becomes an affine overapproximation of $\Phi^{-1}(\cdot)$, \eqref{eq:Num_quantile_reform} is a tightening of \eqref{eq:quantile_final} and Assumption \ref{assm:1} ensures the convexity of \eqref{eq:Num_quantile_reform}. 
%This allows for solutions to be found with an off-the-shelf conic solver.

\begin{lem} \label{lem2}
For a controller $\overline{U}_1, \dots, \overline{U}_v$, if there exists risk allocation variables $\underline{\omega}_{ij}$ and $\underline{\beta}_{ij}$, 
% for its equivalent representation, 
and slack variables $\underline{s}_{ij}$ satisfying \eqref{eq:Num_quantile_reform}, then $\overline{U}_1, \dots, \overline{U}_v$ asymptotically satisfies \eqref{prob:initial_eq_prob_constraints} as $h \rightarrow 0$ and $n \rightarrow \infty$.
\end{lem}

\begin{figure*}[ht]
\emph{Proposed reformulation to solve Problem~\ref{prob:initial}:} 
\begin{subequations}\label{prob:second}
    \begin{align}
        \underset{\substack{\overline{U}_1, \dots, \overline{U}_v \\ \overline{\omega}_{11}, \dots, \overline{\omega}_{ij} \\ \overline{s}_{11}, \dots, \overline{s}_{ij}}}{\mathrm{minimize}} \quad & J\left(
       \boldsymbol{X}_1, \ldots, \boldsymbol X_v,  \overline{U}_1, \dots, \overline{U}_v\right) \\
        \mathrm{subject\ to} \quad  & \overline{U}_1, \dots, \overline{U}_v \in  \mathcal U^N,  \\
        & \text{Dynamics \eqref{eq:lin_dynamics} with initial states }\overline{x}_1(0), \dots, \overline{x}_v(0) 
    \end{align}
    \vspace{-20pt}
     \begin{alignat}{6}
        \forall j \in \mathbb N_{[1,v]}, i \in \mathbb N_{[1,n_j]} \quad & f_i(\overline{x}_j(0), \overline{U}_j) \leq c_i - \frac{1}{g_i}\left(\underline{s}_{ij}\right) & &\text{\quad and/or \quad} && f_i(\overline{x}_j(0), \overline{U}_j) \geq c_i + \frac{1}{g_i}\left(\underline{s}_{ij}\right)  \label{eq:prob2_reform}\\
        \hspace{15pt}\forall j \in \mathbb N_{[1,v]}, i \in \mathbb N_{[1,n_j]}, q \in \mathbb{N}_{[1,l^\ast]} \quad & \underline{s}_{ij} \geq \underline{m}_{ijq} \, \underline{\omega}_{ij} + \underline{c}_{ijq} &&\text{\quad and/or \quad} && \underline{s}_{ij} \geq \underline{m}_{ijq} \, \underline{\beta}_{ij} + \underline{c}_{ijq} \\
        & \textstyle\sum_{j=1}^{v}\sum_{i=1}^{n_j} \underline{\omega}_{ij} \leq \alpha &&\text{\quad and/or \quad} && \textstyle\sum_{j=1}^{v}\sum_{i=1}^{n_j} \underline{\beta}_{ij} \leq \alpha  \\
        \forall j \in \mathbb N_{[1,v]}, i \in \mathbb N_{[1,n_j]} \quad & \underline{\omega}_{ij} \geq 0 &&\text{\quad and/or \quad} && \underline{\beta}_{ij} \geq 0
    \end{alignat}
\end{subequations}
\rule{\textwidth}{0.5pt}
\end{figure*}

\begin{proof} 
By \eqref{eq:taylor_error}, the approximation error $\epsilon \rightarrow 0$ as $h \rightarrow 0$ and $n \rightarrow \infty$. In this case, \eqref{eq:Num_quantile_reform} conservatively enforces \eqref{eq:quantile_final} by \eqref{eq:quantile_over_error}. By Lemma \ref{lem1}, \eqref{prob:initial_eq_prob_constraints} is conservatively enforced. 
\end{proof}

We note that a limitation of our approach is that we can only guarantee constraint satisfaction in the limit.  
% , we wanted to take a moment to discuss how this effects real-world applicability. 
In practice, a sufficiently differentiable distribution will likely behave well enough that four or more derivatives will result in an approximation with small errors given a small enough step size. Many common distributions will fall into this category, especially those of the exponential family of distributions. Where this methodology will likely fail is multi-modal distributions or distributions that have a non-smooth terminating derivative.  We have found empirically that a step size, $h$, on the order of $10^{-6}$, is sufficiently small that the approximation error, \eqref{eq:taylor_error_eq}, is also on the order of $10^{-6}$. 
%In Section \ref{ssec:results-cauchy}, we demonstrate how well the numerical approximation performs on a Cauchy disturbed system in comparison to it's analytical form. 

\subsection{Reverse Convex Constraints}\label{sub:dcp}
% The constraint specification in \eqref{eq:quantile_reform_final_4} is an example of a reverse convex constraint with respect to the variable $\overline{U}$. A reverse convex constraint is the complement of a convex constraint and takes the form $f(x) \geq c$ where $f: \R \rightarrow \R$ is a convex function and $c$ is a scalar. Optimization problems containing reverse convex constraints can be cast as difference of convex optimization problems. The general outline of a difference of convex program is,
A standard approach to handling reverse convex constraints is difference of convex programming, 
\begin{equation} \label{eq:dc}
\begin{split}
    \underset{x}{\mathrm{minimize}} \quad & \mathcal{F}_0(x)-\mathcal{G}_0(x)  \\
    \mathrm{subject\ to} \quad  & \mathcal{F}_i(x)-\mathcal{G}_i(x) \leq 0 \quad \text{for } i \in \mathbb{N}_{[1,L]}  \\
\end{split}   
\end{equation}
in which the cost and constraints are represented as the difference of two convex functions, i.e., $\mathcal{F}_0, \mathcal{F}_i(\cdot): \R^n \rightarrow \R$ and $\mathcal{G}_0, \mathcal{G}_i(\cdot): \R^n \rightarrow \R$ for $x \in \R^n$ are convex.
% where all $L$ constraints and the cost function are represented by the difference of convex functions, $h_0, h_i(\cdot): \R^n \rightarrow \R$ and $j_i(\cdot): \R^n \rightarrow \R$ for $x \in \R^n$. 
The convex-concave procedure solves (\ref{eq:dc}) to a local minimum \cite{boyd_dc_2016} through an iterative approach, which employs first order approximations of $\mathcal{G}_0, \mathcal{G}_i$ at each iteration.  
Feasibility of \eqref{eq:dc} is dependent on the feasibility of the initial conditions.

We can show that \eqref{eq:quantile_reform_final_4} elicits a difference of convex formulation by subtracting $f_i(\overline{x}_j(0),\overline{U}_j)$ from both sides.
% , and consider a first-order approximation of $h_i$ and $j_i$ in the iterations. 
% Since feasible initial conditions are not always known, 
We also add slack variables to accommodate potentially infeasible initial conditions \cite{boyd_dc_2016}, \cite{horst2000}. 
When using a difference of convex program, Lemma \ref{lem1} guarantees a feasible but locally optimal solution.

% Difference of convex programs can be difficult to solve to a global optimum and can be very sensitive to initial conditions. Here, 
% We employ the convex-concave procedure to solve Problem \ref{prob:initial} to a local minimum \cite{boyd_dc_2016}. This procedure is performed by solving \eqref{eq:dc} iteratively with convex approximations of $h_i$ and $j_i$, until two successive iterations meet some convergence criteria or a maximum number of iterations is reached. We consider a first-order approximation of $h_i$ and $j_i$ when constructing the necessary convex approximations. 

% Any feasible initial condition will preserve feasibility. 
% By adding the scaled sum of slack variables to the cost function and increasing the scale each iteration, the program will become more feasible with each iteration.

% The optimization problem \eqref{prob:second} employs Lemma \ref{lem2} to solve Problem \ref{prob:initial}. Since \eqref{eq:prob2_reform} contains reverse convex constraints, it is a difference of convex program. For reverse convex constraint we replace $f_i(\overline{x}_j(0), \overline{U}_j)$ with its first-order approximation in \eqref{eq:prob2_reform}. In doing so we have reformulated all constraints to be either convex or linear. By iterating over \eqref{prob:second} and updating our first-order approximation at each iteration, we solve Problem \ref{prob:initial} to a local optima. 

\section{Experimental Results}\label{sec:results}

We demonstrate our algorithms in simulation on a multi-vehicle spacecraft navigation problem with two disturbances: one that is Gaussian (for validation), and one that is Cauchy (to demonstrate our method's capabilities).  
All computations were done on a 1.80GHz i7 processor with 16GB of RAM, using MATLAB, CVX \cite{cvx} and Gurobi \cite{gurobi}. Polytopic construction and plotting was done with MPT3 \cite{MPT3}. The system formulations are implemented in SReachTools \cite{sreachtools}. All code is available at \url{https://github.com/unm-hscl/shawnpriore-approximate-quantiles}.

Consider a scenario in which three satellites are stationed in low earth orbit.
%$850$ \si{\km} above the Earth's surface. 
% We seek to perform a proximity operation to re-position the satellites. 
Each satellite is tasked with reaching a terminal target set, while avoiding other satellites.   
The relative dynamics of each spacecraft, with respect to a known and fixed origin, are described by the Clohessy-Wilthire-Hill (CWH) equations
% equations model the dynamics of a  satellite in an elliptical orbit about a known and fixed origin 
\cite{wiesel1989_spaceflight}
\begin{subequations}
\begin{align}
\ddot x - 3 \omega^2 x - 2 \omega \dot y &= \frac{F_x}{m_c} \label{eq:cwh:a}\\
\ddot y + 2 \omega \dot x & = \frac{F_y}{m_c} \label{eq:cwh:b}\\
\ddot z + \omega^2 z & = \frac{F_z}{m_c}. \label{eq:cwh:c}
\end{align}   
\label{eq:cwh}
\end{subequations}
with input $u_i = [ \begin{array}{ccc} F_x & F_y & F_z\end{array}]^\top$, mass $m_c$, and orbital rate $\omega = \sqrt{\frac{\mu}{R^3_0}}$, with gravitational constant $\mu$ and orbital radius $R_0$. 
We discretize (\ref{eq:cwh}) with a first-order hold, with 
% The continuous time dynamics are discretized via a first-order hold with 
sampling time $30$\si{\s}, and insert a disturbance process that captures model uncertainties, so that dynamics for vehicle $i$ are described by  
% Each satellite is assumed to follow a linear, time-invariant representation of the discretized dynamics about a common origin,
\begin{equation}
    \boldsymbol x_i(k+1) = A \boldsymbol x_i(k) + B \overline{u}_i(k) + \boldsymbol{w}_i(k) \label{eq:cwh_lin}
\end{equation}
with 
%state $\boldsymbol{x}_i(k) \in \R^n$, input $\overline{u}_i(k) \in \mathcal{U} \subset \R^m$, disturbance $\boldsymbol{w}_i(k) \in \R^n$ is a independent distributed real random variable with known distribution, time index $k \in \mathbb{N}_{[0,N]}$, and satellite index $i$. We assume the initial states $\overline{x}_i(0)$ are known and the set $\mathcal{U}$ is a convex polytope. Here, the 
% state vector $\boldsymbol{x}_i(k)$ representing the position and and velocities of the satellite along the major axes and the input vector $\overline{u}_i(k)$ represents the forces applied along the major axes. Actuation limits, $\mathcal{U}$, are constrained as the set $[-5,5]$ for each input dimension. 
$\mathcal{U}_i = [-5,5]^3$, and time horizon $N=8$, corresponding to 4 minutes of operation.
% We set the total number of time steps to 8, totaling a 4 minute operation.

The terminal sets $\mathcal T_i$ are $5\times 5 \times 5$m boxes centered around desired terminal locations in $x,y,z$ coordinates, with speeds bounded in all three directions by $[-0.01, 0.01]$m/s.  For collision avoidance, we presume that all satellites must remain at least $r=15$m away from each other, hence $S = \begin{bmatrix} I_{3} & 0_{3} \end{bmatrix}$ to extract the positions.  Violation thresholds for terminal sets and collision avoidance are $\alpha_{\mathcal T} = \alpha_r = 0.1$, respectively. 
\begin{align}
    \Prob \left\{ \bigcap_{i=1}^3 \boldsymbol x_i(N) \in \mathcal{T}_i \right\} &\geq 1-\alpha_{\mathcal T} \label{eq:terminal}\\
    \Prob \left\{ \bigcap_{k=1}^N \bigcap_{i,j=1}^{3} \left\| S \!\cdot\!\left(\boldsymbol x_i(k)\! - \! \boldsymbol x_j(k)\right)\right\| \geq r \right\} &\geq 1-\alpha_r  \label{eq:avoidance}
\end{align}

The performance objective is based on fuel consumption. 
\begin{equation}
    J(\overline{U}_1, \overline{U}_2, \overline{U}_3) = \sum^3_{i=1} \overline{U}_i^\top \overline{U}_i
\end{equation}

When approximating the numerical quantile, we presume  intervals $h=5\times 10^{-6}$, and maximum approximation error $\xi = 0.1$. For symmetric distributions, we set the instantiating point, $p_0$, to $0.5$ with known quantile $\Phi^{-1}(p_0) = 0$. For non-symmetric distributions, we set the instantiating point to $0.9$. Computation of $\Phi^{-1}(p_0)$ was completed with MATLAB's implementation of the incomplete gamma function for the Chi quantile. Analytical results were used to compute $\Phi^{-1}(p_0)$ for the sum of squared Cauchy random variables. Each quantile approximation used the first three derivatives of the pdf. Convergence criteria was defined as the difference of sequential outputs and the sum of slack variables both less than $10^{-8}$; difference of convex programs were limited to 100 iterations. 

\subsection{6D CWH with a Gaussian Disturbance}\label{ssec:results-normal}
% To demonstrate the efficacy of our proposed method, we start with the 6-d CWH system as defined by \eqref{eq:cwh}-\eqref{eq:cwh_lin}. We add a Gaussian noise parameterized with a zero mean vector and covariance matrix,
% \begin{equation}
%     \Sigma = \begin{bmatrix} 10^{-4} \times I_3 & 0_{3\times3} \\ 0_{3\times3} &  5 \times  10^{-8} \times I_3\end{bmatrix}
% \end{equation} 
We first consider a Gaussian noise with zero mean and covariance $\Sigma = \mathrm{diag}(10^{-4} \cdot I_3, 5 \times  10^{-8} \cdot I_3)$.
% The main challenge is that quantiles arising from collision avoidance and the target set constraints do not have an analytical form.  
Once reformulated, the target set constraint has a Gaussian distribution and the collision avoidance has a Chi distribution with three degrees of freedom. As  
neither have an analytical expression for their quantile function, the use of standard tools or methods \cite{vinod2019piecewise,PrioreACC21} for Gaussian distributions is not viable.
% . Chi distributions with 3 or more degrees of freedom don't have an analytical quantile functions and are problematic for a quantile based approach. 

% We turn our attention to the terminal set and collision avoidance constraints. Without loss of generality, a convex polytope can be written as the intersection of constraints, 
% \begin{equation}
% \begin{split}
%     \mathcal{T}_j &\equiv \{\boldsymbol{x}_j(k) | P \,\boldsymbol{x}_j(k) \leq \overline{q} \} \\
%     & \equiv \bigcap_{i=1}^{N}\left\{\boldsymbol{x}_j(k) \mid P_{i,\cdot} \, \boldsymbol{x}_j(k) \leq \overline{q}_{i} \right\}
% \end{split}    
% \end{equation}
% where $N$ is the number of half-spaces defining the polytope, and $P_{j, \cdot}$ and $\overline{q}_{i}$ define the boundary and direction of the half-space. Since the state variable is random we modify the equality constraint to the form of \eqref{eq:prob_constraints_convex}-\eqref{eq:prob_constraints_reverse_convex},
% \begin{equation}
% \begin{array}{rrl}
%     &P_{i,\cdot}  \boldsymbol{x}_j(k) &\leq \overline{q}_{i} \\
%     \Leftrightarrow& P_{i,\cdot} \left(A^k \overline{x}_j(0) + \mathcal{C}_u(k) \overline{U}_j + \mathcal{C}_w(k) \boldsymbol{W}\right) &\leq \overline{q}_{i} \\
%     \Leftrightarrow & \underbrace{P_{i,\cdot}\left( A^k \overline{x}_j(0) + \mathcal{C}_u(k) \overline{U}_j \right)}_{f_i(\overline{x}_j(0), \overline{U}_j)} + \underbrace{P_{i,\cdot} \,\mathcal{C}_w(k) \boldsymbol{W}}_{\eta_i} &\leq \underbrace{\overline{q}_{i}}_{c_i} 
% \end{array}
% \end{equation}

For terminal constraint (\ref{eq:terminal}), formulation into (\ref{eq:prob_constraints_convex}) via \eqref{eq:demo1} results in 
% Note that 
$\eta_i$ that is a univariate Gaussian distribution with zero mean and variance $P_{i,\cdot} \, \Sigma(k) \, P_{i,\cdot}^\top$, with $\Sigma (k) = \sum_{i=0}^{k-1}(A^i)^\top \Sigma A^i$, such that
% . To simplify computation across time steps, we consider the the function 
$g_i \eta_i  = P_{i,\cdot} \, \Sigma(k) \, P_{i,\cdot}^\top \eta_i$. 
% This allows for the use of the standard Gaussian distribution for each time step and satellite. 

For the collision avoidance constraint (\ref{eq:avoidance}), formulation into (\ref{eq:prob_constraints_reverse_convex}) via \eqref{eq:demo2} follows the derivation as in \cite[Thm 1]{PrioreACC21}, and results in
% Theorem 1 of \cite{PrioreACC21} allows us to handle the collision avoidance constraint. We summarize the reformulation as follows,
\begin{equation}
 \begin{split}
    &\left\| S\!\left(\boldsymbol x_i(k) - \boldsymbol x_j(k)\right)\right\| \\
    \equiv &\left\| S\!\left(A^k \overline{x}_{i-\!j}(0) + \mathcal{C}_u(k) \overline{U}_{i-\!j} + \mathcal{C}_w(k) \boldsymbol{W}_{i-\!j} \right)\right\| \\
    \geq &\left\| S\!\left(A^k \overline{x}_{i-\!j}(0) + \mathcal{C}_u(k) \overline{U}_{i-\!j}\right)\right\| - \left\| S \mathcal{C}_w(k) \boldsymbol{W}_{i-\!j} \right\|  \\
    \geq & \underbrace{\left\| \vphantom{\left(2S \Sigma(k) S^\top \right)^{\frac{1}{2}}} S\!\left(A^k \overline{x}_{i-\!j}(0) + \mathcal{C}_u(k) \overline{U}_{i-\!j}\right)\right\|}_{f_i(\overline{x}(0), \overline{U})} \!-\!\underbrace{ \left\| \left(2S \Sigma(k) S^\top \right)^{\frac{1}{2}} \boldsymbol \rho \right\|}_{g_i  \eta_i}
\end{split}   
\end{equation}
where the index $\cdot_{i-j}$ represents the difference between the two variables, respectively, and $\boldsymbol \rho$ is a multivariate Gaussian. By the compatibility of matrix norms, we obtain 
\begin{equation} 
g_i  \eta_i = \left\| \left(2S \Sigma(k) S^\top \right)^{\frac{1}{2}} \right\| \cdot \left\| \boldsymbol \rho \right\|
\end{equation}
where $\left\| \boldsymbol \rho \right\|$ 
%, the norm of a 3-dimensional multivariate Gaussian, 
% Here, the random variable is the norm of a standard multivariate Gaussian of three dimensions. It 
follows a Chi distribution with three degrees of freedom.

We compare the proposed method with the mixed integer particle approach using two polytopic overapproximations of the collision avoidance constraint, based on the $L_{\infty}$ and $L_1$ norm.
% . We approximate the collision avoidance area as mixed integer linear constraints with the $L_{\infty}$ and $L_1$ norm, as represented by a cube and a octahedron, respectively. 
% The size of each polytope was chosen such that the polytope would circumscribe the sphere representing the $L_2$ norm. 
We generated 10 disturbance samples to generate an open-loop controller. Note that different disturbance samples were used when generating the controller for either variant.

The resulting trajectories, costs, and computation times differ dramatically, as shown in
Figure \ref{fig:normal_trajectory} and Table \ref{table:results_time_gaussian}.
% show that the three methods produce differing results with dramatic differences in computation time.  
(Differences in the $z$-coordinates were minimal, so we only plot the $x$ and $y$ coordinates in Figure \ref{fig:normal_trajectory}.)  
To assess constraint satisfaction, we generated $10^5$ Monte-Carlo sample disturbances for each approach; the 
% is examined in Figure \ref{fig:normal_distance} compares 
the $L_2$ distance between the mean positions at each time step are shown in Figure \ref{fig:normal_distance}.
% For each method, we generated $10^5$ Monte-Carlo sample disturbances to compare constraint satisfaction. 
Table \ref{table:results_constraints_gaussian} shows that while all three methods satisfied the collision avoidance constraint, neither particle control approach satisfied the terminal set constraint.

The proposed method performed two to three orders of magnitude faster than particle control. Given the significant increase in binary variables needed to perform particle control, this comes as no surprise. We attempted to increase the number of disturbance samples, however, we could not generate a solution in a under two hours. Conversely, the low number of disturbance samples is likely the cause for the poor performance with respect to the target set constraint. Given the random nature of the sampling process, ten samples is not enough to characterize the behaviour on a larger scale. 

Figures \ref{fig:normal_trajectory} and \ref{fig:normal_distance} show that the differences in avoidance regions impacted the results. With the $L_2$ collision avoidance region overapproximated by both the $L_{\infty}$ and $L_1$ regions, we

\begin{figure}[H]
    \includegraphics[width=0.95\linewidth]{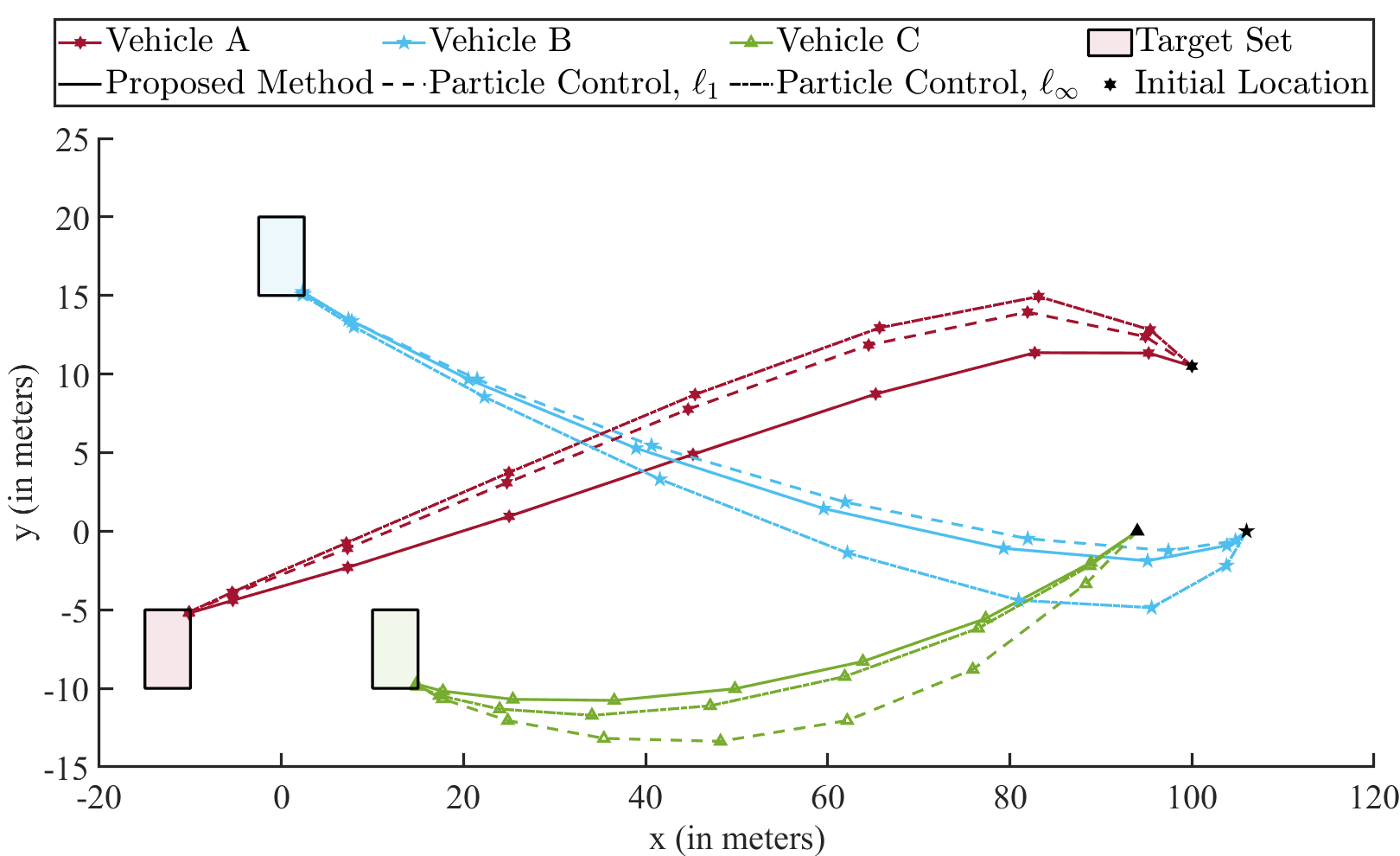}
    \centering
    \caption{Comparison of trajectories in $(x,y)$ coordinates from proposed method (solid) and particle control (dashed for $L_1$ norm; dotted for $L_\infty$ norm).}
    \label{fig:normal_trajectory}
\end{figure}

\begin{figure}[H]
    \includegraphics[width=0.95\linewidth]{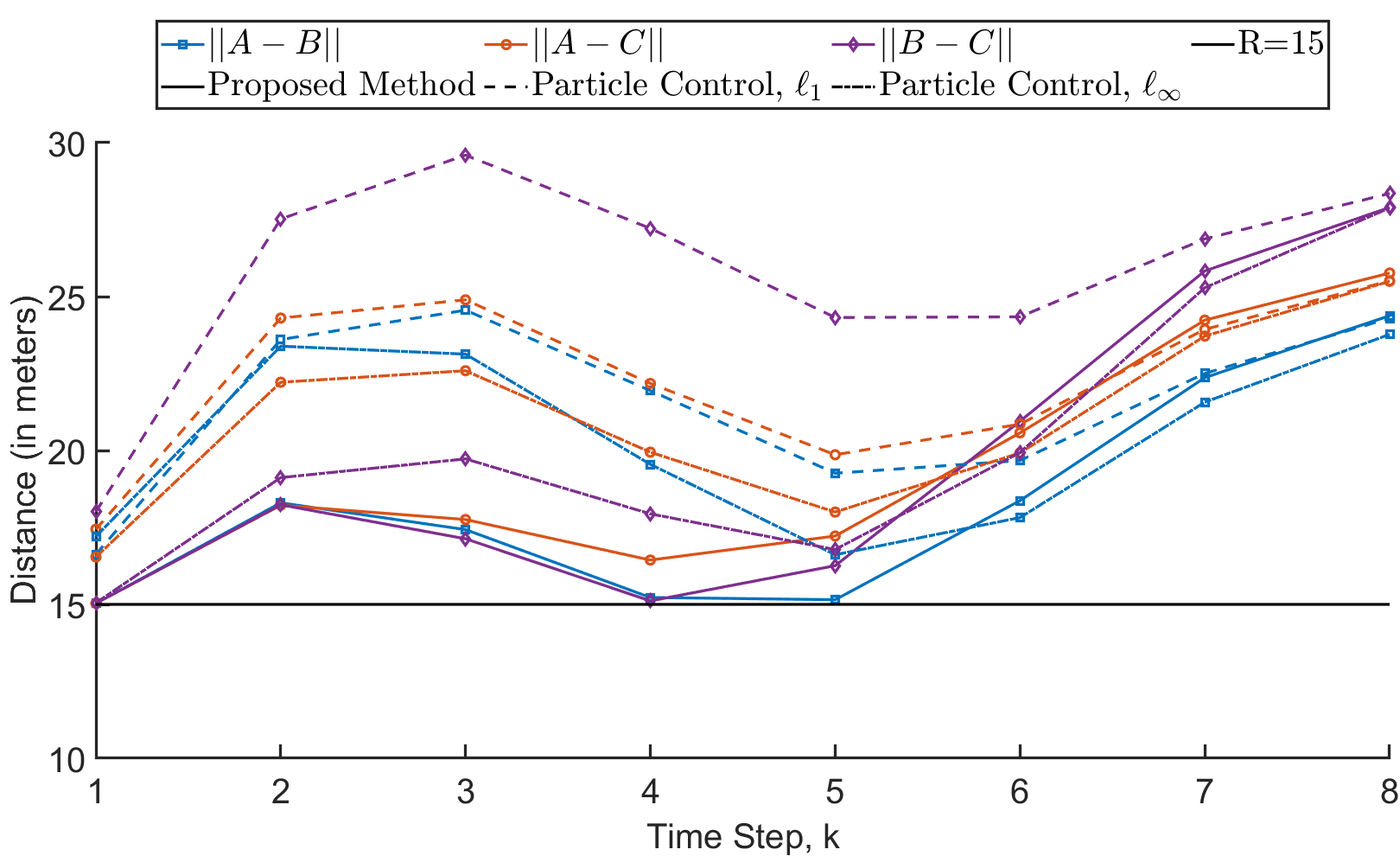}
    \centering
    \caption{Comparison of $L_2$ inter-satellite distances between proposed method (solid) and particle control (dashed for $L_1$ norm; dotted for $L_\infty$ norm).}
    \label{fig:normal_distance}
\end{figure}

\begin{table}[H]
\caption{Computation Time and Control Cost for CWH Dynamics with a Gaussian Disturbance.}
\centering
\begin{tabular}{|l||c|c|c|}
\hline
\multirow{2}{*}{Metric} & \multirow{2}{*}{\shortstack{Proposed \\ method}} & \multicolumn{2}{c|}{Particle control}\\
\cline{3-4}
& & $L_\infty$ & $L_1$ \\
\hline \hline
Computation Time (sec)      & 6.82          & 245.65   & 4199.10\\
\hline
$J(\overline{U}_1, \overline{U}_2, \overline{U}_3)$  & 92.04       & 102.79 & 117.56 \\
\hline
\end{tabular}
\label{table:results_time_gaussian}
\end{table}

\begin{table}[H]
\centering
\caption{Constraint Satisfaction (''SAT'') for CWH dynamics with Gaussian Disturbance, with $10^5$ Samples and Probabilistic Violation Threshold of $1-\alpha =0.9$.}
\begin{tabular}{|l||c|c|c|c|c|c|}
\hline
\multirow{2}{*}{Constraint} &  \multirow{2}{*}{\shortstack{Proposed \\ method}} & \multirow{2}{*}{SAT} & \multicolumn{4}{c|}{Particle control}\\
\cline{4-7}
& & & $L_\infty$ & SAT & $L_1$ & SAT\\
\hline \hline
\multirow{2}{*}{\shortstack[l]{Collision\\Avoidance}} & \multirow{2}{*}{0.9630} & \multirow{2}{*}{$\checkmark$} & \multirow{2}{*}{0.9997} & \multirow{2}{*}{$\checkmark$} & \multirow{2}{*}{1.0000} & \multirow{2}{*}{$\checkmark$} \\ 
&&&&&&\\
\hline
\multirow{2}{*}{\shortstack[l]{Terminal\\Set}} & \multirow{2}{*}{0.9127} & \multirow{2}{*}{$\checkmark$} & \multirow{2}{*}{0.2183} &  & \multirow{2}{*}{0.0993} & \\
&&&&&&\\
\hline
\end{tabular}
\label{table:results_constraints_gaussian}
\end{table}

\noindent expected the collision avoidance likelihood to be significantly higher than the proposed method. However, the sharp edges of the polytopes created control choices that led to more aggressive direction changes. This phenomena is apparent in the lack of smoothness in the particle control trajectories in Figure \ref{fig:normal_trajectory}. Similarly, the larger avoidance regions effectively increased the avoidance distance to $18\si{\m}$ for the $L_{\infty}$ particle control run and $23\si{\m}$ for the $L_{1}$ particle control run, as observed in Figure \ref{fig:normal_distance}. The additional distance had a distinct impact on the overall cost of each of these controllers.  To produce a closer comparison, we also used 14- and 26-faced polytopes, however neither resulted in a solution within a $24$ hour time frame. 

\subsection{4D CWH with Cauchy Disturbance} \label{ssec:results-cauchy}

We consider the planar CWH dynamics (\ref{eq:cwh:a}), (\ref{eq:cwh:b}) with a Cauchy disturbance that is parameterized with location as zero and scale elements $\Gamma(k)$,
\begin{equation}
    \Gamma(k)_j = 
    \begin{cases} 
        10^{-4} & \text{if } j \in \{4n + \{1,2\} \mid  n \in \mathbb{N}_{[0,k-1]}\} \\
        5\! \times\! 10^{-8} & \text{if } j \in \{4n+\{3,4\} \mid  n \in \mathbb{N}_{[0,k-1]}\}
    \end{cases}
\end{equation}
corresponding to position and velocity elements, respectively.

For the terminal set constraint, because the set is axis-aligned, it can be written as function of a single Cauchy random variable,

\begin{figure}[H]
    \includegraphics[width=0.925\linewidth]{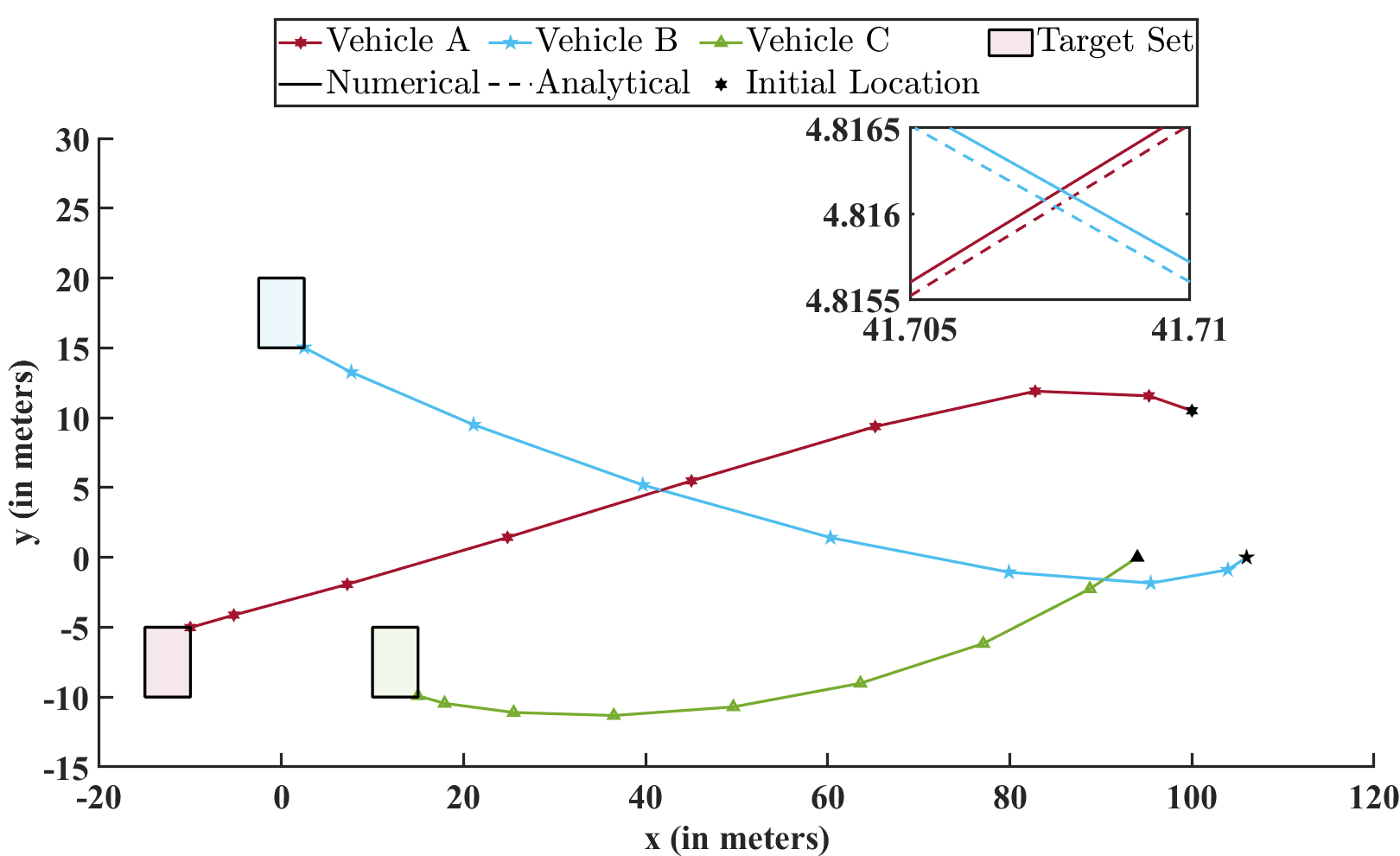}
    \centering
    \caption{Comparison of trajectories in $(x,y)$ coordinates for proposed method (solid) and with an analytic quantile (dashed). The trajectories are nearly indistinguishable as seen in the magnified subplot.}
    \label{fig:cauchy_trajectory}
\end{figure}

\begin{figure}[H]
    \includegraphics[width=0.925\linewidth]{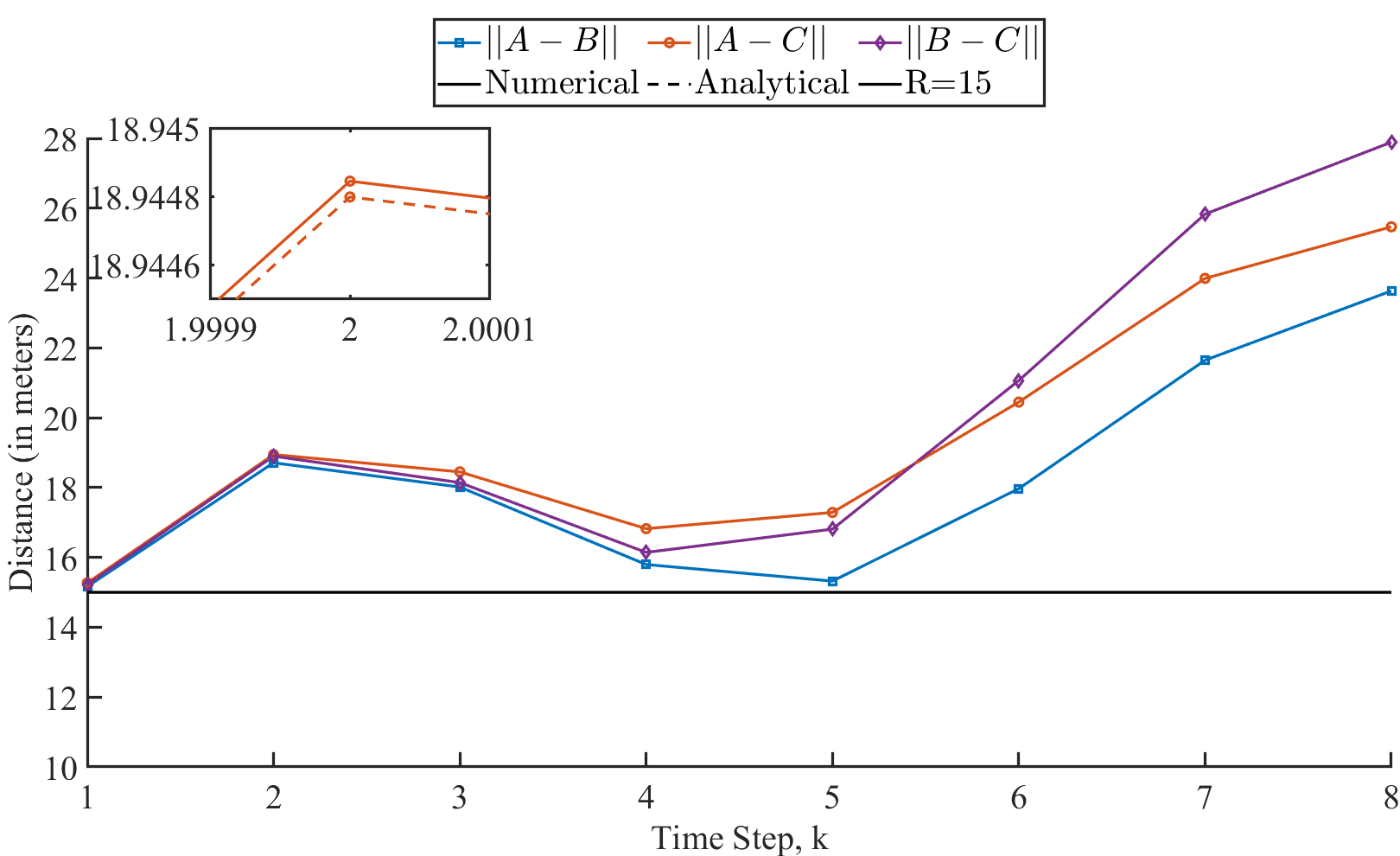}
    \centering
    \caption{Comparison of inter-satellite distances between proposed method (solid) and with an analytic quantile (dashed). The inter-satellite distances are nearly indistinguishable as seen in the magnified subplot.}
    \label{fig:cauchy_distance}
\end{figure}

\noindent \begin{equation}\label{eq:cauchy_term}
\setlength\arraycolsep{2pt}
\begin{array}{rrl}
    & \overline{e}_{m}^\top \, \boldsymbol{x}_j &\leq \overline{q}_{i} \\
    \Leftrightarrow & \underbrace{ \overline{e}_{m}^\top \left(A^k \overline{x}_j(0) +  \mathcal{C}_u(k) \overline{U}_j \right)}_{f_i(\overline{x}_j(0), \overline{U}_j)} \!+\! \underbrace{ \, \overline{e}_{m}^\top  \,\mathcal{C}_w(k) \boldsymbol{W}}_{g_i \eta_i} &\leq \underbrace{\overline{q}_{i}}_{c_i} 
\end{array}
\end{equation}
where $\overline{e}_m$ is the $m$-th column of an appropriately sized identity matrix. Here,
\begin{equation}
    g_i \eta_i = \overline{e}_{m}^\top\mathcal{C}_w(k) \Gamma(k)  \eta_i \\
\end{equation} 
and $\eta_i$ has a standard Cauchy distribution. Note that without a direct covariance measure between two Cauchy distributions, \eqref{eq:cauchy_term} cannot be easily extended to account for polytopic constraints of more than one dimension.

For the collision avoidance constraint, the main challenge arises from the coupling across random variables that arises from taking a norm. By falsely considering each dimension as independent, we under approximate the norm
\begin{align}
    &\left\| S\!\left(\boldsymbol x_i(k) - \boldsymbol x_j(k)\right)\right\| \\
    \equiv &\left\| S\!\left(A^k \overline{x}_{i-\!j}(0) \!+\! \mathcal{C}_u(k) \overline{U}_{i-\!j} + \mathcal{C}_w(k) \boldsymbol{W}_{i-\!j} \right)\right\|  \nonumber\\
    \geq &\left\| S\!\left(A^k \overline{x}_{i-\!j}(0) \!+\! \mathcal{C}_u(k) \overline{U}_{i-\!j}\right)\right\| \!-\! \left\| S \mathcal{C}_w(k) \boldsymbol{W}_{i-\!j} \right\| \nonumber \\
    % \geq & \left\| S\!\left(A^k \overline{x}_{i-\!j}(0) \!+\! \mathcal{C}_u(k) \overline{U}_{i-\!j}\right)\right\| \!-\! \left\| \mathrm{max}\left(S \mathcal{C}_w(k) \Gamma(k) \right) \boldsymbol \rho \right\| \\
    \geq & \underbrace{\left\| S\!\left(A^k \overline{x}_{i-\!j}(0) \!+\! \mathcal{C}_u(k) \overline{U}_{i-\!j}\right)\right\|}_{f_i(\overline{x}(0), \overline{U})} \!-\! \underbrace{ \mathrm{max}\left(S \mathcal{C}_w(k) \Gamma(k) \right) \left\| \boldsymbol \rho \right\|}_{g_i  \eta_i} \nonumber
\end{align}
where $\boldsymbol \rho$ is a 2D vector consisting of independent and identically distributed standard Cauchy variables, and $\mathrm{max}(\cdot)$ returns the element of the argument vector with the maximum value. 
Note that the random variable of interest is $\left\| \boldsymbol \rho \right\|^2$, which we can show through convolution to have closed-form expressions for the pdf, cdf, and the quantile,

\noindent 
\begin{subequations}
\begin{align}
    \phi_{\left\| \boldsymbol \rho \right\|^2}(x) &= \frac{2}{\pi \sqrt{1+x}(2+x)} \\
    \Phi_{\left\| \boldsymbol \rho \right\|^2}(x) &= \textstyle\frac{4}{\pi}\arctan\left(\sqrt{1+x}\right) - 1\\
    \Phi^{-1}_{\left\| \boldsymbol \rho \right\|^2}(p) &= \tan^2 \left( \textstyle\frac{\pi}{4}(1+p)\right) -1
\end{align}
\end{subequations}

Figures \ref{fig:cauchy_trajectory} and \ref{fig:cauchy_distance}, and Tables \ref{table:results_time_cauchy} and \ref{table:results_constraints_cauchy}, show that the proposed method with the numerical quantile performed nearly identically to the proposed method with an analytical quantile. We observed differences on the order of $10^{-4}$, as shown in the subplot of Figure \ref{fig:cauchy_trajectory}.  This is likely attributed to

\begin{table}[H]
\caption{Computation Time and Control Cost for CWH Dynamics with Cauchy Disturbance.}
\centering
\begin{tabular}{|l||l|l|}
\hline
\multirow{2}{*}{Metric} & \multicolumn{2}{c|}{Proposed method}\\
\cline{2-3}
&  Numerical  & Analytical\\
\hline \hline
Computation Time (sec)  & 153.86   & 157.98\\
\hline
$J(\overline{U}_1, \overline{U}_2, \overline{U}_3)$ & 93.11 & 93.11 \\
\hline
\end{tabular}
\label{table:results_time_cauchy}
\end{table}

\begin{table}[H]
\centering
\caption{Constraint Satisfaction (``SAT'') for CWH dynamics with Cauchy Disturbance, with $10^5$ Samples and Probabilistic Violation Threshold of $1-\alpha = 0.9$.}
\begin{tabular}{|l||c|c|c|c|}
\hline
\multirow{2}{*}{Constraint} & \multicolumn{4}{c|}{Proposed method}\\
\cline{2-5}
 & Numerical & SAT & Analytical & SAT\\
\hline \hline
Collision Avoidance & 0.9979 & $\checkmark$ & 0.9978 & $\checkmark$  \\
\hline
Terminal Set & 0.9086 & $\checkmark$ & 0.9085 & $\checkmark$ \\
\hline
\end{tabular}
\label{table:results_constraints_cauchy}
\end{table}

\noindent our choice of a very small interval $h$, which yielded a highly accurate quantile approximation.

\section{Conclusion}\label{sec:conclusion}
We proposed a method for chance constrained stochastic optimal control of LTI systems that exploits a numerical approximation of the quantile function.  Our approach is amenable to distributions whose pdfs are sufficiently smooth.  We demonstrated our approach on a multi-vehicle satellite control problem with Gaussian and Cauchy disturbances.

\appendix

\subsection{Reformulation of Constraints}
\label{sec:reformulation}

\subsubsection{Polytopic constraint set}
Consider a terminal set constraint, captured by \eqref{eq:constraint_t} as
% Consider the target set constraint, \eqref{eq:constraint_t}. Let's assume some vehicle has a terminal set constraint, $\mathcal{T}_N$, with probabilistic violation threshold, $1-\alpha$. Then we must meet the following constraint,
\begin{equation}
    \Prob \{\boldsymbol x(N)  \in  \mathcal T_N \} \geq  1-\alpha,
\end{equation}
whose argument can be rewritten in halfspace form as
% \begin{equation}
% \begin{split}
%     \mathcal{T}_j &\equiv \{\boldsymbol{x}_j(k) | P \,\boldsymbol{x}_j(k) \leq \overline{q} \} \\
%     & \equiv \bigcap_{i=1}^{N}\left\{\boldsymbol{x}_j(k) \mid P_{i,\cdot} \, \boldsymbol{x}_j(k) \leq \overline{q}_{i} \right\}
% \end{split}    
% \end{equation}
% 
% \begin{equation}
%     P \boldsymbol x_j (N) \leq \overline{p}
% \end{equation}
$P \boldsymbol x_j (N) \leq \overline{p}$
for some $P \in \R^{L \times n}$, $\overline{p} \in \R^L$, for $L$ the number of half-space constraints in $\mathcal{T}_k$.  Expanding into an intersection of scalar constraints, we obtain
\begin{equation}
\label{eq:demo1}
\setlength\arraycolsep{2pt}
\begin{array}{rrl}
    &P  \boldsymbol{x}_j(k) &\leq \overline{p} \\
    \Leftrightarrow & \displaystyle \bigcap_{i=1}^L P_{i,\cdot}  \boldsymbol{x}_j(k) &\leq \overline{p}_{i} \\
    % \Leftrightarrow& \displaystyle \bigcap_{i=1}^L P_{i,\cdot} \left(A^k \overline{x}_j(0) + \mathcal{C}_u(k) \overline{U}_j + \mathcal{C}_w(k) \boldsymbol{W}\right) &\leq \overline{p}_{i} \\
    \Leftrightarrow \!& \!\displaystyle \bigcap_{i=1}^L \underbrace{P_{i,\cdot}\left( A^k \overline{x}_j(0) + \mathcal{C}_u(k) \overline{U}_j \right)}_{f_i(\overline{x}_j(0), \overline{U}_j)} + \underbrace{P_{i,\cdot} \,\mathcal{C}_w(k) \boldsymbol{W}}_{g_i \eta_i}  &\leq \underbrace{\overline{p}_{i}}_{c_i} 
\end{array}
\end{equation}

% By linearity of \eqref{eq:lin_dynamics} and \eqref{eq:constraint_t},
%\begin{subequations} 
%\begin{align}
%    P \boldsymbol x_j (N) = & P \left(A^N \overline{x}_j(0) + %\mathcal{C}_u(N) \overline{U}_j + \mathcal{C}_w(N) \boldsymbol{W}\right) \\
    % = & P \left[A^N \overline{x}_j(0) + \mathcal{C}_u(N) \overline{U}_j\right] + P \mathcal{C}_w(N) \boldsymbol{W} \\
    %\equiv & \bigcap_{i=1}^L \underbrace{P_{i,\cdot} \!\left(A^N \overline{x}_j(0) + \mathcal{C}_u(N) \overline{U}_j\right)}_{f_i(\overline{x}_j(0), \overline{U}_j)} \!+\! \underbrace{P_{i,\cdot} \mathcal{C}_w(N) \boldsymbol{W}}_{g_i \eta_i} \label{eq:demo1}
%\end{align}
%\end{subequations}
where $P_{i,\cdot}$ is the $i^{\text{th}}$ row of the matrix $P$, meaning that
% Here, $f_i(\overline{x}_j(0), \overline{U}_j)$ is linear, and therefore, convex. Thus, we can write an equivalent constraint to \eqref{eq:constraint_t} as,
\begin{equation}
    % \begin{array}{rrl}
         \Prob \{\boldsymbol x(N)  \in  \mathcal T_N \}  \geq  1-\alpha 
        \Leftrightarrow \: \Prob \: \left\{\eqref{eq:demo1}\right\}  \geq  1-\alpha
    % \end{array}
\end{equation}
% Here, the two way implication would allow for a very close approximation to the optimal solution. Note that norm constraints of a target region are also amenable to the form \eqref{eq:prob_constraints_convex} via the triangle inequality. This may provide computational advantages due to reduced problem dimension but is likely less accurate as a result of the bound property of the triangle inequality. 
as in (\ref{eq:prob_constraints_convex}), with 
% Here, the 
random variable $\eta$ that is a linear transformation of $w$.

\subsubsection{Norm based constraint set}
% Next, consider the collision avoidance constraint, \eqref{eq:constraint_r}. We assume two controlled vehicles has a known initial states, $\overline{x}_i(0)$ and $\overline{x}_j(0)$, and some disturbances, $\boldsymbol W_{i}$ and $\boldsymbol W_{j}$. Let the minimum $L_2$ distance required be $r$ with probabilistic violation threshold, $1-\alpha$. Then the constraint for some time-step $k$ is
For probabilistic collision avoidance between vehicles $i$ and $j$ with minimum $L_2$ distance $r \in \mathbb R_+$ and violation threshold $1-\alpha$, we consider  
\begin{equation}
    \Prob\{\|\boldsymbol x_i(k) - \boldsymbol x_j(k)\| \geq r\} \geq  1-\alpha 
\end{equation}
Using the reverse triangle inequality, we obtain
\begin{subequations}\label{eq:demo2}
\begin{align}
    &\|\boldsymbol x_i(k) - \boldsymbol x_j(k)\| \label{eq:demo_2_init} \\
    =&\|A^k[\overline{x}_i(0) \! - \! \overline{x}_j(0)] \! + \! \mathcal{C}_u(k) [\overline{U}_{ i} \! - \! \overline{U}_{ j}] \! + \! \mathcal{C}_w(k) (\boldsymbol W_{ i}\! - \! \boldsymbol W_{ j})\|\\
    % =&\|A^k[\overline{x}_i(0) \! - \! \overline{x}_j(0)] \! + \! \mathcal{C}_u(k) [\overline{U}_{ i} \! - \! \overline{U}_{ j}] \! - \! \mathcal{C}_w(k) (\boldsymbol W_{ j}\! - \! \boldsymbol W_{ i})\|\\
    \geq&\underbrace{\|A^k[\overline{x}_i(0) \! - \! \overline{x}_j(0)] \! + \! \mathcal{C}_u(k) [\overline{U}_{ i} \! - \! \overline{U}_{ j}] \|}_{f_i(\overline{x}_j(0), \overline{U}_j)} \! - \! \underbrace{\|\mathcal{C}_w(k) (\boldsymbol W_{ i}\! - \! \boldsymbol W_{ j})\|}_{g_i \eta_i} \label{eq:demo_2_final}
\end{align}
\end{subequations}
% where $f_i(\overline{x}_j(0), \overline{U}_j)$ is convex due to the convexity of norms. Now we can create a one way implication on the satisfaction of \eqref{eq:constraint_r} as such:
hence
\begin{equation}
% \begin{array}{rrl}
    \Prob\{\|\boldsymbol x_i(k) - \boldsymbol x_j(k)\| \geq r\}  \geq  1-\alpha \\
    \Leftarrow \Prob\{\eqref{eq:demo_2_final} \geq r\}  \geq 1-\alpha
% \end{array}
\end{equation}
% When transforming constraints of the form \eqref{eq:constraint_o}, we replace the initial state of the vehicle, $\overline x_j(0)$, with with the initial state of the object, $\overline o_j(0)$, and remove the input, $\overline{U}_j$, as the object is not controlled. 
as in (\ref{eq:prob_constraints_reverse_convex}). Note that this is a one-way implication. 

\bibliography{refs}
\bibliographystyle{ieeetr}
\end{document}